\documentclass[11pt]{amsart}
\usepackage{amsmath,amssymb}
\numberwithin{equation}{section}
 \usepackage{xcolor}

\newtheorem{theorem}{Theorem}[section]

\newtheorem{lemma}[theorem]{Lemma}
\newtheorem{lem}[theorem]{Lemma}
\newtheorem{prop}[theorem]{Proposition}
\newtheorem{coro}[theorem]{Corollary}
\newtheorem{corollary}[theorem]{Corollary}

\theoremstyle{definition}
\newtheorem{definition}[theorem]{Definition}
\newtheorem{defi}[theorem]{Definition}

\newtheorem{remark}[theorem]{Remark}
 \newtheorem*{ackn}{Acknowledgements}
 
 \newtheorem*{thmA}{Theorem A} 
 \newtheorem*{thmB}{Theorem B}

 \newcommand{\R}{\mathbb R}
 
 \newcommand{\C}{\mathbb C}
  
 \newcommand{\N}{\mathbb N}

 \newcommand{\e}{\varepsilon}

 \newcommand{\f}{\varphi}
 
 \newcommand{\p}{\psi}

 \usepackage{hyperref}
\hypersetup{
    unicode=false,        
    pdftoolbar=true,      
    pdfmenubar=true,       
    pdffitwindow=false,     
    pdfstartview={FitH},    
    pdftitle={Hessian equations},    
    pdfauthor={Guedj, Lu},     
    colorlinks=true,       
   linkcolor=blue,          
    citecolor=blue,        
    filecolor=black,      
    urlcolor=blue}



\begin{document}

\title[Degenerate complex Hessian equations]{Degenerate complex Hessian equations on compact Hermitian manifolds}
\author{Vincent Guedj \& Chinh H. Lu}

\address{Institut de Math\'ematiques de Toulouse   \\ Universit\'e de Toulouse \\
118 route de Narbonne \\
31400 Toulouse, France\\}

\email{\href{mailto:vincent.guedj@math.univ-toulouse.fr}{vincent.guedj@math.univ-toulouse.fr}}
\urladdr{\href{https://www.math.univ-toulouse.fr/~guedj}{https://www.math.univ-toulouse.fr/~guedj/}}

\address{Univ Angers, CNRS, LAREMA, SFR MATHSTIC, F-49000 Angers, France.}

\email{\href{mailto:hoangchinh.lu@univ-angers.fr}{hoangchinh.lu@univ-angers.fr}}
\urladdr{\href{https://math.univ-angers.fr/~lu/}{https://math.univ-angers.fr/~lu/}}
\date{\today}

 \begin{abstract}
  In this note we provide uniform a priori estimates for solutions to degenerate complex Hessian equations
  on compact hermitian manifolds.
  Our approach relies on the corresponding a priori estimates for Monge-Amp\`ere equations; it provides an extension as well as a short
  alternative proof to results of Dinew-Ko{\l}odziej, Ko{\l}odziej-Nguyen and Guo-Phong-Tong.
\end{abstract}

 \maketitle

 \begin{center}
{\em \`A la m\'emoire de Jean-Pierre Demailly.}
\end{center}

\vskip1cm


\section*{Introduction}

 Let $X$ be a compact complex manifold equipped with a hermitian form $\omega_X$.
  We fix $1 \leq k \leq n=\dim_{\C} X$, $\omega$ another hermitian form, 
 and  $\mu_X$ a volume form on $X$.
 The complex Hessian equation is 
 $$
 (\omega+dd^c \f)^k \wedge \omega_X^{n-k}=c \mu_X,
 $$
 where $c>0$ is a positive constant 
  and $\f$ (the unknown) is a smooth $(\omega,\omega_X,k)$-subharmonic function,
 i.e.   $(\omega+dd^c \f)^j \wedge \omega_X^{n-j} \geq 0$  for all $1 \leq j \leq k$.
 
 When $k=1$ this reduces to the Laplace equation. When $k=n$ this yields the complex
 Monge-Amp\`ere equation which was famously solved by Yau \cite{Yau78} when $\omega$ is K\"ahler,
 and by Tosatti-Weinkove \cite{TW10} in the hermitian setting.
 
 From now on we assume that $1<k<n$. A unique solution   has been provided
 by Dinew-Ko{\l}odziej \cite{DK17} when  $\omega=\omega_X$  is K\"ahler; the hermitian setting has
 been resolved independently by Sz\'ekelyhidi \cite{Szek18} and Zhang \cite{Zh17}.
 
 \smallskip
 
It is quite natural to consider such equations in degenerate settings, 
allowing the measure $\mu_X=f\omega_X^n$ to be merely absolutely continuous with respect
to the volume form $\omega_X^n$.
Such degenerate complex Hessian equations
have been intensively studied in the past fifteen years (see
\cite{Blo05,Lu13,DK14,DL15,LN15,Lu15,KN16,Char16,GN18,DDT19,BZ20,GP22} and the references therein).

In this note we  go one step further by allowing the 
form $\omega$ to degenerate as well:
we no longer assume it is positive, we merely ask that it is {\it big},
in analogy with the corresponding notion from algebraic geometry \cite{BEGZ10},
i.e. we assume there is an $\omega$-psh function $\rho$
with analytic singularities such that 
$\omega+dd^c \rho$ dominates a hermitian form.
We let $\Omega$ denote the Zariski open set where $\rho$ is smooth.

One can not expect any longer that $(\omega,\omega_X,k)$-subharmonic functions are smooth in $X$, but we 
show in Section \ref{sec:defoperator} how to define the Hessian operator
$H_k(f)=(\omega+dd^c \f)^k \wedge \omega_X^{n-k}$ for such functions
$\f$ which are {\it continuous} in $\Omega$ and globally bounded on $X$.

\smallskip

In this context we obtain the following:

\begin{thmA}{\it 
Let $\omega$ be a semi-positive and big form
and $f \in L^p(\omega_X^n)$ with $p>n/k$ and $\int_X f \omega_X^n>0$.
There exists $c>0$ and $\f$  a bounded $(\omega,\omega_X,k)$-subharmonic function, continuous in $\Omega$, such that
 $$
 (\omega+dd^c \f)^k \wedge \omega_X^{n-k}=c f \omega_X^n.
 $$}
\end{thmA}

This result extends the main result of \cite{KN16}.
It applies in particular to the case when $\pi:X \rightarrow Y$ is the desingularization
of a compact hermitian variety $Y$, and $\omega=\pi^* \omega_Y$ is the pull-back
of a hermitian form $\omega_Y$ on $Y$.

\smallskip

As is often the case in similar contexts, the main difficulty lies in establishing uniform a priori estimates.
This is the content of our second main result:

\begin{thmB}{\it 
Let $\omega_0,\omega$ be  semi-positive and big forms  such that $\omega_0 \leq \omega$.
Let $\f$ be a smooth $(\omega,\omega_X,k)$-subharmonic function solution to 
 $$
 (\omega+dd^c \f)^k \wedge \omega_X^{n-k}=f \omega_X^n,
 $$
where  $f \in L^p(dV_X)$, $p>n/k$.
 Then
 $$
 {\rm Osc}_X(\f) \leq C,
 $$
 where $C>0$ depends on $\omega_0,n,k,p$ and an upper bound for
 $||f||_{L^p}$.}
\end{thmB}

The approach of Ko{\l}odziej-Nguyen relies on a  pluripotential theory for degenerate complex Hessian equations on hermitian manifolds (mimicking the corresponding results obtained by these authors 
and Dinew for the Monge-Amp\`ere equation \cite{DK12}, \cite{KN15}, \cite{KN19}, \cite{Ng16}).
An alternative PDE proof has been recently provided by Guo-Phong-Tong-Wang in \cite{GPT21,GPTW21}
when $X$ is K\"ahler, and by Guo-Phong \cite{GP22} when $X$ is hermitian and $\omega$ is a hermitian form.

We provide in this note a   proof which relies directly on the corresponding results for the complex Monge-Amp\`ere operator.
Given $\f$ a solution of the $k$-Hessian equation with density $f$, we show that 
the $\omega$-psh envelope of $\f$ solves a Monge-Amp\`ere equation whose density
 is bounded from above by $f^{n/k}$.
A lower bound on $\f$ thus follows from the $L^{\infty}$-a priori estimates for the Monge-Amp\`ere 
equation established in \cite{GL21}. 

Our method is new and  also of interest when $X$ is K\"ahler;
it applies to more general densities (see Remark \ref{rem:optimalbound}); in particular the 
results of  Theorem A and  Theorem B
hold when $f$ merely satisfies
$$
\int_X f^{\frac{n}{k}} |\log f|^n (\log |\log f|)^{n+\delta} <+\infty,
$$
for some $\delta>0$.
Radial examples show that this condition is almost optimal.

\begin{ackn} 
This work is partially supported by the research project HERMETIC 
(ANR-11-LABX-0040), as well as by the ANR project PARAPLUI.
\end{ackn}

\section{Preliminaries}

In the whole article we let $(X,\omega_X)$ denote a compact complex manifold 
of complex dimension $n \in \N^*$,
equipped with a fixed hermitian form $\omega_X$. We use real differential operators $d= \partial + \bar{\partial}$, $d^c= i (\bar \partial -\partial)$, so that $dd^c= 2i \partial \bar{\partial}$. 

\smallskip

We also fix a semi-positive $(1,1)$-form $\omega$ which is {\it big}:

  \begin{definition}
    We say that $\omega$ is big if there exists an $\omega$-psh function $\rho$
    such that $\omega+dd^c \rho$ dominates a hermitian form.
        \end{definition}

 It follows from Demailly's regularization theorem \cite{Dem94} that one can further assume that
 $\rho$ has analytic singularities, in particular $\omega+dd^c \rho$ is a 
 hermitian form in some Zariski open set
 that we usually denote by $\Omega$ in what follows.
 
 \smallskip
 
 An important source of examples of big forms is provided by the following construction:
 if $V$ is a compact complex space endowed with a hermitian form $\omega_V$, and 
$\pi:X \rightarrow V$ is a resolution of singularities, then $\omega=\pi^* \omega_V$ is big.

 \subsection{$(\omega,\omega_X,k)$-subharmonic functions}
 
 \subsubsection{Definition}
 
 Let $U\subset X$ be an open subset and $\gamma$  a hermitian form in $U$.
  A function $h\in {\mathcal C}^2(U,\mathbb R)$ is called $\gamma$-harmonic if 
 \[
 dd^c h \wedge \gamma^{n-1} = 0 \; \text{pointwise in }\; U. 
 \]
 We let $H_{\gamma}(U)$ denote the set of all $\gamma$-harmonic functions in $U$.  A function $u : U \rightarrow \mathbb R \cup \{-\infty\}$ is $\gamma$-subharmonic if it is upper semicontinuous and for every $D\subset U$ and every $h\in H_{\gamma}(D)$ the following implication holds
 \[
 u \leq h \; \text{on} \; \partial D \Longrightarrow u \leq h \; \text{in}\; D. 
 \]
 As shown in \cite[Section 9]{HL13} 
 the  $\gamma$-subharmonicity can be equivalently defined using viscosity theory:
  $u$ is $\gamma$-subharmonic in $U$ if and only if the inequality $dd^c u\wedge \gamma^{n-1}$ holds in the viscosity sense in $U$. The latter condition means that for every $x_0\in U$ and 
  every ${\mathcal C}^2$ function $\chi$ defined in a neighborhood $V$ of $x_0$, 
  the following implication holds: 
 \[
 \left[u(x) \leq \chi(x), \; \forall x \in V\;  \text{and}\;  u(x_0)=\chi(x_0) \right] \Longrightarrow dd^c \chi \wedge \gamma^{n-1}\geq 0 \; \text{at} \; x_0. 
 \]
When $u\in {\mathcal C}^2(U)$, it  is $\gamma$-subharmonic if and only if $dd^c u \wedge \gamma^{n-1}\geq 0$ pointwise.

\begin{defi}
We let $SH_{\gamma}(U)$ denote the set of all $\gamma$-subharmonic functions in $U$ that are locally integrable.	
\end{defi}

  If $u\in SH_{\gamma}(U)$ then, by  \cite[Theor\`eme 1, p.136]{HH72}, $dd^c u \wedge \gamma^{n-1}\geq 0$ in the sense of distributions. Conversely, by  \cite[Theorem A]{Lit83}, the latter condition implies that $u$ coincides 
  a.e. with a unique function $\hat{u}\in SH_{\gamma}(U)$ defined by 
  \[
  \hat{u}(x) := {\rm ess} \limsup_{y\to x} u(y) := \lim_{r\searrow 0} {{\rm ess}\sup }_{B(x,r)} u,
  \]  
  where ${{\rm ess}\sup }_{B(x,r)} u$ is the essential supremum of $u$ in the ball $B(x,r)$.
   
   \begin{defi}
 	A function $\f:U \rightarrow \R \cup \{-\infty\}$ is quasi-$\gamma$-subharmonic if locally $\f= u+\rho$ where $u$ is $\gamma$-subharmonic and $\rho$ is smooth. 
 	
 	We say that a function $\f$ is $(\omega,\gamma)$-subharmonic if it is quasi-$\gamma$-subharmonic and  $(\omega+dd^c \f) \wedge \gamma^{n-1}\geq 0$ in the sense of distributions on $X$. 
 	
 	We let $SH_{\gamma}(U,\omega)$ denote the set of all $(\omega,\gamma)$-subharmonic functions in $U$.
   \end{defi}
   
 If $u\in {\mathcal C}^2(U,\mathbb R)$ then $u\in SH_{\gamma}(U,\omega)$ if and only if $(\omega+dd^c u)\wedge \gamma^{n-1} \geq 0$ pointwise. In general this inequality holds in the viscosity sense:

 \begin{lemma}\label{lem:max principle linear}
 If $\varphi\in SH_{\gamma}(U,\omega)$ and $\chi$ is a ${\mathcal C}^2$ function such that $\varphi-\chi$ attains a local maximum at $x_0\in U$, then $(\omega+dd^c \chi)\wedge\gamma^{n-1}\geq 0$ at $x_0$.  	
  \end{lemma}
  
  \begin{proof}
  	Fix $a>1$. Let $\rho$ be a quadratic function such that $dd^c \rho= a \omega(x_0)$. By continuity of $\omega$, we have $dd^c \rho \geq \omega$ in a small ball $B$ around $x_0$. The function $\rho+\varphi$ is then $\gamma$-subharmonic in $B$ and the function $\rho+\varphi -(\rho+\chi)$ attains a local maximum at $x_0$. It thus follows that, at $x_0$, 
  	\[
  	(a \omega+dd^c \chi)\wedge \gamma^{n-1}= dd^c (\rho+ \chi) \wedge \gamma^{n-1}\geq 0.
  	\] 
 The conclusion follows by letting $a\searrow 1$.
  \end{proof}

  Fix an integer $1\leq k\leq n$, and let $\Gamma_k(\omega_X)$ denote the set of all $(1,1)$-forms $\alpha$ such that 
  \[
  \alpha^j \wedge \omega_X^{n-j} > 0, \; j=1,...,k. 
  \]
  
By G{\aa}rding's inequality \cite{Gar59}, if $\alpha, \alpha_1,...,\alpha_{k-1} \in \Gamma_k(\omega_X)$ then
  \[
  \alpha \wedge \alpha_1\wedge ... \wedge \alpha_{k-1} \wedge \omega_X^{n-k} > 0. 
  \]
  In particular, $\alpha_1\wedge ... \wedge \alpha_{k-1} \wedge \omega_X^{n-k}$ is a strictly positive $(n-1,n-1)$-form on $X$. By \cite{Mic82}, one can write $\alpha_1\wedge ... \wedge \alpha_{k-1} \wedge \omega_X^{n-k}= \gamma^{n-1}$, for some hermitian form $\gamma$. We let $F_k(\omega_X)$ denote the set of all such hermitian forms.  
  
  \begin{definition}
  	A function $\f: U \rightarrow \R \cup \{-\infty\}$ is $(\omega, \omega_X,k)$-subharmonic if it is $(\omega,\gamma)$-subharmonic in $U$ for all $\gamma\in F_k(\omega_X)$. 
  	
  	We let $SH_{\omega_X}(U,\omega,k)$ denote the set of all $(\omega, \omega_X,k)$-subharmonic functions on $U$ which are locally integrable in $U$.
  \end{definition}
  
    If $u$ is of class ${\mathcal C}^2$ in $U$ then, by \cite{Gar59}, $u$ is $(\omega,\omega_X,k)$-subharmonic if and only if $\omega+dd^c u \in \Gamma_k(\omega_X)$.
We refer the reader to \cite{Blo05} and \cite[Section 2.3]{KN16} for basic properties of  $(\omega,\omega_X,k)$-sh functions; they easily extend to our    context.

\smallskip

  Since $\omega_X$ is fixed once and for all in this note, we
 let $SH(X,\omega,k)$ denote the set of $(\omega,\omega_X,k)$-sh functions;  they form a nested sequence
 $$
 SH(X,\omega)  \supset \cdots SH(X,\omega,k-1)  \supset SH(X,\omega,k) \supset \cdots \supset 
PSH(X,\omega),
 $$
 where $SH(X,\omega)=SH(X,\omega,1)$ is the set of $\omega$-sh functions, while
  $PSH(X,\omega)$ $=SH(X,\omega,n)$
 is the set of $\omega$-psh functions, for which $\omega+dd^c \f \geq 0$ is a positive
 current of bidegree $(1,1)$ (see \cite[Chapter 8]{GZbook}).
 
 \smallskip
 
Observe that $SH(X,\omega,k) \subset SH(X,\omega',k)$ if $\omega \leq \omega'$. 
 Since $\omega\leq A \omega_X$ for some $A>0$, any $(\omega, \omega_X, k)$-sh function is 
 also $(A\omega_X,1)$-sh. We thus have the following compactness result 
 (see \cite[Lemma 3.3]{KN16}, \cite[Section 9]{GN18}). 
 
 \begin{prop} \label{pro:integrability}
 The set of normalized $(\omega,\omega_X,k)$-subharmonic functions is compact in the $L^1$-topology. In particular, there exists $C>0$ such that for all $\f \in SH(X,\omega,k)$ normalized by 
 $\sup_X \f=0$,
 $$
 \int_X (-\f) \omega_X^n \leq C.
 $$
 \end{prop}

 \subsubsection{Smooth approximation}

 \begin{theorem}\label{thm: approximation}
 If $\omega$ is hermitian, any $(\omega,k)$-subharmonic function  is the pointwise limit
 of a  decreasing sequence of smooth $(\omega,k)$-subharmonic functions.
 \end{theorem}
 
This result is a slight extension of  \cite[Lemma 3.20]{KN16}, which itself extends  \cite[Theorem 1.2]{LN15} from the K\"ahler to the hermitian setting.
We propose below an alternative and more direct proof.
 
 \begin{proof}
 	Let $u$ be a smooth function on $X$ and let $f$ be a smooth positive density such that 
 	\[
 	(\omega+dd^c u)^k \wedge \omega_X^{n-k} \leq f \omega_X^n. 
 	\]
 	For $j\geq 1$, we solve the complex Hessian equation 
 	\[
 	(\omega +dd^c \varphi_j)^k \wedge \omega_X^{n-k} = e^{j (\varphi_j-u)} f\omega_X^n. 
 	\]
 	As shown in \cite{KN16}, the estimates established in \cite{Szek18} 
 	provide a smooth solution $\varphi_j$ which is $(\omega,k)$-subharmonic on $X$. 
 	We want to prove that $\varphi_j$ converges uniformly to $P_{\omega,k}(u)$,  
 	the largest $(\omega,k)$-subharmonic function lying below $u$. 
 	The theorem follows from this fact, as 
 	shown in \cite{LN15,KN16}.
 	
 	The classical maximum principle ensures that $\varphi_j\leq u$,
 	 hence $\varphi_j \leq P_{\omega,k}(u)$. 
 	Fix now $\varphi\in SH(X,\omega,k)$ such that $\varphi\leq u$.  Let $x_0\in X$ be a point where the function $\varphi -(1+1/j)\varphi_j$ attains its maximum on $X$. By definition $\varphi$ belongs to $SH_{\gamma}(X,\omega)$, where $\gamma$ is the hermitian form such that 
 	\[
 	\gamma^{n-1} = (\omega+dd^c \varphi_j)^{k-1}\wedge \omega_X^{n-k}. 
 	\]
 	From now on the computations are made at $x_0$.
 	It follows from Lemma \ref{lem:max principle linear} that
 	\[
 	(\omega+(1+1/j) dd^c \varphi_j)\wedge \gamma^{n-1}\geq 0. 
 	\]
 	From this and G{\aa}rding's inequality \cite{Gar59} we obtain
 	\begin{flalign*}
 		\frac{j+1}{j} (\omega+dd^c \varphi_j)^k\wedge \omega_X^{n-k} &\geq \frac{1}{j} \omega \wedge  (\omega+dd^c \varphi_j)^{k-1}\wedge \omega_X^{n-k} \\
 		&\geq \frac{1}{j} e^{\frac{j(k-1)}{k} (\varphi_j-u)} f^{(k-1)/k} g^{1/k} \omega_X^n,
 	\end{flalign*}
 	where $g\omega_X^n = \omega^k \wedge \omega_X^{n-k}$. Let $b\in (0,1)$ be a positive constant such that $g\geq b^k f$ on $X$. The above inequality yields
 	\[
 	(1+j)e^{j(\varphi_j-u)/k} \geq b,
 	\]
 	hence
 	\[
 	\varphi_j - u \geq \frac{k (\log b - \log (1+j))}{j}.
 	\]
 	Thus, for some uniform constant $C\geq 1$ we have 
 	\[
 	\left(1+\frac{1}{j}\right)\varphi_j - u \geq \left(1+\frac{1}{j}\right)(\varphi_j-u) + \frac{u}{j} \geq  -C \frac{\log j}{j}. 
 	\]
 	Since $\varphi\leq u$, we also obtain
 	\[
 	\varphi-\left(1+\frac{1}{j}\right) \varphi_j  \leq  C \frac{\log j}{j}. 
 	\]
 	
 	This inequality holds on $X$, as  	$\varphi-(1+1/j) \varphi_j$ attains its global maximum at $x_0$.
 	 Since this is true for all $\varphi \in SH(X,\omega,k)$ lying below $u$, we obtain
 	\[
 	P_{\omega,k}(u) -C' \frac{\log j}{j}  \leq  \varphi_j  \leq P_{\omega,k}(u),
 	\] 	
 	for a uniform constant $C'$,
 	proving that $\varphi_j$ converges uniformly to $P_{\omega,k}(u)$.   
 	 \end{proof}

 \subsection{The Hessian operator} \label{sec:defoperator}
 
 \subsubsection{Hermitian forms}
 We assume in this subsection that $\omega$ is a hermitian form.
 For $(\omega,\omega_X,k)$-subharmonic functions $\varphi$ of 
 class ${\mathcal C}^2$, the complex Hessian operator is defined by 
 \[
  H_k(\varphi) = (\omega+dd^c \f)^k \wedge \omega_X^{n-k}.
 \]
By Theorem \ref{thm: approximation}, one can pointwise approximate from above any $(\omega,\omega_X,k)$-sh function $\f$ by a decreasing sequence of smooth $(\omega,\omega_X,k)$-sh functions $\f_j$. 

 If $\varphi$ is moreover continuous on $X$, then the convergence is uniform, and it was shown in \cite{KN16} that the sequence of Hessian measures $H_k(\varphi_j)$ weakly converges to a unique positive Radon measure, independent of the approximants. 
 One sets
 $$
  H_k(\f):=\lim_{j \rightarrow +\infty} H_k(\f_j) := (\omega+dd^c \f)^k \wedge \omega_X^{n-k}.
 $$
 It is likely that one can define these operators  for $(\omega,\omega_X,k)$-sh functions
 that are merely bounded. This is indeed the case when $\omega_X$ is closed, as one can proceed by induction and use integration by parts following Bedford-Taylor's construction \cite{BT76, BT82}. 
 The situation is however delicate in the general hermitian setting, and more subtle than for the
 Monge-Amp\`ere operator.
 
 \smallskip
 
 Among the basic properties enjoyed by this operator are the following:
 \begin{itemize}
 \item $\f \mapsto H_k(\f)$ is continuous for the uniform topology;
 \item if $u_1,\ldots, u_k \in \mathcal{SH}(X,\omega,k)$ are continuous, the mixed Hessian operators
 $(\omega+dd^c u_1) \wedge \cdots \wedge (\omega+dd^c u_k) \wedge \omega_X^{n-k}$ are similarly  defined and continuous;
 \item if $u_1,\ldots, u_k$ are smooth 
and $(\omega+dd^c u_j)^k \wedge \omega_X^{n-k}=f_j dV_X$ then
\begin{equation} \label{eq:mixed}
(\omega+dd^c u_1) \wedge \cdots \wedge (\omega+dd^c u_k) \wedge \omega_X^{n-k}
 \geq (f_1 \cdots f_k)^{1/k} dV_X.
\end{equation}
 \end{itemize}
These are a consequence of inequalities due to G{\aa}rding. 

\subsubsection{Big forms}

We now no longer assume that $\omega$ is positive, but rather that it is semi-positive and big. Fix a quasi-plurisubharmonic function $\rho: X \rightarrow \mathbb{R}\cup \{-\infty\}$ with analytic singularities such that $\omega+dd^c \rho \geq \delta \omega_X$, for some constant $\delta>0$. 
We let $\Omega$ denote the Zariski open set where $\rho$ is smooth.

If  $\varphi\in \mathcal{SH}(X,\omega,k) \cap L^{\infty}(X)$ is continuous in $\Omega$,
it follows from \cite[Proposition 2.9]{KN16} that
for each smooth form $\gamma$, the mixed product
 $(\omega_X+dd^c \varphi)^j \wedge \gamma^{l-j} \wedge \omega_X^{n-l}$
  is a well-defined Radon measure in $\Omega$ with finite mass.  

\begin{defi}
For $\varphi\in \mathcal{SH}(X,\omega,k) \cap L^{\infty}(X) \cap {\mathcal C}^0(\Omega)$ we set
$$
H_k(\f):= 
\sum_{j=0}^k \binom{k}{j} (\omega_X+dd^c \varphi)^j \wedge (\omega-\omega_X)^{k-j}\wedge\omega_X^{n-k}. 
$$
\end{defi}

This operator is well-defined in $\Omega$ with finite total mass, 
so we consider its trivial extension to $X$.
 The definition is motivated by the identity
\begin{flalign*}
	(\omega+dd^c \varphi)^k \wedge \omega_X^{n-k} &= (\omega_X +dd^c \varphi + \omega-\omega_X)^k \wedge\omega_X^{n-k}\\
	&= \sum_{j=0}^k \binom{k}{j} (\omega_X+dd^c \varphi)^j \wedge (\omega-\omega_X)^{k-j}\wedge\omega_X^{n-k}. 
\end{flalign*}

The operator $H_k$ is continuous along sequences that converge locally uniformly in $\Omega$
\cite[Proposition 2.11]{KN16}. The following result follows.

\begin{prop}
If $\varphi\in \mathcal{SH}(X,\omega,k) \cap L^{\infty}(X) \cap {\mathcal C}^0(\Omega)$,
then $H_k(\f)$ is a {positive} Radon measure.
\end{prop}

\begin{proof}
 Fix $\varepsilon>0$ and approximate $\varphi$ 
 -using Theorem \ref{thm: approximation}- by a decreasing sequence  $(\varphi_j)$ of $(\omega+\varepsilon\omega_X)$-sh functions. 
 Since $(\omega+\varepsilon \omega_X+dd^c \varphi_j)^k \wedge \omega_X^{n-k} \geq 0$,
 we infer from \cite[Proposition 2.11]{KN16} that 
 $(\omega+\varepsilon \omega_X+dd^c \varphi)^k \wedge \omega_X^{n-k} \geq 0$.
 The conclusion follows by letting $\varepsilon\to 0$, as we then obtain that $(\omega+dd^c \varphi)^k \wedge \omega_X^{n-k}\geq 0$.   
\end{proof}


\begin{lem} \label{lem:Demkey}
Let $\f,\p$ be $(\omega,\omega_X,k)$-sh functions which are continuous in $\Omega$
and bounded  on $X$. Then 
\[
H_k( \max(\varphi,\psi)) 
 \geq {\bf 1}_{\{\varphi>\psi\}} H_k(\f) + {\bf 1}_{\{\psi\geq \varphi\}} H_k(\p).
\]
In particular, if $\varphi\leq \psi$ then 
\[
 1_{\{\p = \f \}} H_k(\f) 
 \leq 1_{\{\p = \f \}} H_k(\p).
\]
\end{lem}

\begin{proof}
Fix $t>0$. 
Since $\{\varphi<\psi+t\}$ and $\{\psi+t<\varphi\}$ are open in $\Omega$, we have 
\begin{flalign*}
	H_k(\max(\varphi,\psi+t)) &\geq 
 {\bf 1}_{\{\varphi>\psi+t\}} (\omega+dd^c \varphi)^k \wedge\omega_X^{n-k} + {\bf 1}_{\{\psi+t> \varphi\}} (\omega+dd^c \psi)^k \wedge\omega_X^{n-k}\\
 &\geq  {\bf 1}_{\{\varphi>\psi+t\}} (\omega+dd^c \varphi)^k \wedge\omega_X^{n-k} + {\bf 1}_{\{\psi \geq \varphi\}} (\omega+dd^c \psi)^k \wedge\omega_X^{n-k}. 
\end{flalign*}
Letting $t\searrow 0$ we obtain the first statement of the lemma.
The second one follows straightforwardly from the first.
\end{proof}

 \subsubsection{$SH(X,\omega,k)$-envelopes}

We   consider two types of envelopes in this note.
 
 \begin{defi}
If $h:X \rightarrow \R$ is a continuous function, we set
 \[
 P_{\omega}(h):=\sup \left\{ u \in PSH(X,\omega), \, u \leq h \right\},
 \]
 and $P_{\omega,k}(h):=\sup \left\{ u \in SH(X,\omega,k), \, u \leq h \right\}$.
 \end{defi}

Here is a list of basic properties of these envelopes.
\begin{itemize}
\item The   measure $(\omega+dd^c P_{\omega}(h))^n$ is concentrated on the  set ${\mathcal C}=\{P_{\omega}(h)=h\}$.
\item The  complex Hessian  measure $(\omega+dd^c P_{\omega,k}(h))^k \wedge \omega_X^{n-k}$ is concentrated
 on  the contact set ${\mathcal C}_k=\{P_{\omega,k}(h)=h\}$.
\item If $h$ is ${\mathcal C}^{1,1}$-smooth and $\omega$ is hermitian, then 
$P_{\omega}(h)$ is ${\mathcal C}^{1,1}$-smooth and 
$$
(\omega+dd^c P_{\omega}(h))^n={\bf 1}_{{\mathcal C}} (\omega+dd^c h)^n.
$$
\item If $h$ is ${\mathcal C}^{1,1}$-smooth and $\omega$ is hermitian, then 
$P_{\omega,k}(h)$ is ${\mathcal C}^{1,1}$-smooth and 
$$
(\omega+dd^c P_{\omega,k}(h))^k \wedge \omega_X^{n-k}
={\bf 1}_{{\mathcal C_k}} (\omega+dd^c h)^k \wedge \omega_X^{n-k}.
$$
\end{itemize}

We refer the reader to \cite{Ber19,Tos18,GLZ19,CZ19,CM21} for the proof of these results,
as well as further information on these envelopes.

%
%
%
%

  \subsection{A priori estimates for the complex Monge-Amp\`ere operator}

  We are going to use the following a priori  estimates  established 
  in \cite{GL21}.
  
    \begin{theorem}  \label{thm:uniformHessian}
   Let $\omega,\omega_0$ be semi-positive and big $(1,1)$-forms such that $\omega \geq \omega_0$. 
   Fix $\delta>0$ and 
   $g \in L^{1+\delta}(dV_X)$ 
   a non-negative density.
If $\f \in PSH(X,\omega) \cap L^{\infty}(X)$ satisfies
 $(\omega+dd^c \f)^n \leq g dV_X$, then
 $$
 {\rm Osc}_X (\f) \leq  C
 $$
 for some  $C \in \R^+$ that depends on 
 $n,\delta,\omega_0$ and an upper bound on $||g||_{L^{1+\delta}}$.
 \end{theorem}

 Since this result is not explicitly stated   in \cite{GL21}, we
   briefly sketch its proof.
  
\begin{proof}
We normalize $\f$ so that $\sup_X \f=0$. It follows from
 Skoda's uniform integrability theorem 
\cite[Theorem 8.11]{GZbook} that for $0<\alpha$ small enough,
$$
\int_X e^{-\alpha \frac{(2+\delta)(1+\delta)}{\delta} \f} dV_X \leq C_{\alpha}
$$
is bounded from above independently of $\f$.
Setting $\tilde{g}=e^{-\alpha \f} g$,
H\"older inequality  then shows that 
 $\tilde{g}  \in L^{1+\delta/2}$ with
$$
||\tilde{g}||_{L^{1+\delta/2}} \leq C_{\alpha}^{1/q} ||g||_{L^{1+\delta}}^{1/p},
$$
where $p=\frac{1+\delta}{1+\delta/2}$ and $q=2+2/\delta$.
\cite[Lemma 3.3]{GL21} shows the existence of 
$c(\delta,\omega)>0$ and $u \in PSH(X,\omega)$ such that  $-1 \leq u \leq 0$ and
$$
(\omega+dd^c u)^n \geq c(\delta,\omega) \frac{\tilde{g} dV_X}{||\tilde{g}||_{L^{1+\delta/2}}}
\geq e^{\alpha u} c(\delta,\omega) \frac{\tilde{g} dV_X}{||\tilde{g}||_{L^{1+\delta/2}}}.
$$
The function $v:=u-\alpha^{-1} \log [c(\delta,\omega)/||\tilde{g}||_{L^{1+\delta/2}}] \in PSH(X,\omega)$
thus satisfies
$$
e^{-\alpha v} (\omega+dd^c v)^n = e^{-\alpha v}  (\omega+dd^c u)^n  \geq \tilde{g} dV_X
\geq e^{-\alpha \f} (\omega+dd^c \f)^n.
$$
The comparison principle \cite[Corollary 1.14]{GL21} finally ensures that  $v \leq \f$. 

This provides a uniform lower bound on $\f$, which only depends on
$\delta,\omega$ and an upper bound on $||g||_{L^{1+\delta}}$.
\end{proof}

\section{Domination principle}


Fix  $\omega$ a big form, $\rho$  an $\omega$-plurisubharmonic function with analytic singularities such that $\omega+dd^c \rho \geq  \delta \omega_X$, $\delta>0$,
and set $\Omega:= \{\rho>-\infty\}$.

\smallskip

We fix  a constant $B_1>0$ such that for all $x \in \Omega$,
	\[
	-B_1  \omega_{\rho}^2 \leq dd^c \omega_{\rho} \leq B_1\omega_{\rho}^2
	\; \; \text{ and } \; \; 
	-B_1  \omega_{\rho}^3 \leq d \omega_{\rho} \wedge d^c \omega_{\rho} \leq B_1\omega_{\rho}^3. 
	\]
     The existence of $B_1$ is clear since $d\omega=d\omega_{\rho}$,
$d^c\omega=d^c\omega_{\rho}$  and
 $$
 -B  \omega_{X}^2 \leq dd^c \omega \leq B\omega_{X}^2
\; \; \text{ and } \; \; 
 -B  \omega_{X}^3 \leq d \omega \wedge d^c \omega \leq B\omega_{X}^3
 $$
 for some $B>0$. 
 The following is an extension of \cite[Theorem 1.11]{GL21}.

	 \begin{theorem} \label{thm: non vanishing}
		Assume $\omega$ is big. If $u$ is a bounded 
		$(\omega,\omega_X,k)$-subharmonic function which is continuous in $\Omega$, then $\int_X (\omega+dd^c u)^k \wedge \omega_X^{n-k}>0$. 
 \end{theorem}

 \begin{proof}
 Fix $\rho$ and $B_1$ as above. We set 
 $m=\inf_{\Omega}(u-\rho)$. For $s>0$ we set 
 $$
 \phi = \max(u,\rho+m+s)
 \; \; \text{ and } \; \; 
 U:= \{u<\phi\}= \{u<\rho+m+s\}.
 $$
  Observe that $U$ is relatively compact in 
  $\Omega$, and it is non empty for all $s>0$, by definition of $m$.
  It follows from \cite[Corollary 2.4]{KN16} that, for any $0\leq j\leq k-1$, 
 \[
 dd^c (\omega_u^j \wedge \omega_{\phi}^{k-j-1} \wedge \omega_X^{n-k}) \leq C_1 \sum_{m=k-2}^{k-1}\sum_{l=0}^m  \omega_u^l \wedge \omega_{\phi}^{m-l} \wedge \omega_X^{n-m},
 \]
 where $C_1$ is a uniform constant. 
 On the open set $U$ since $\omega_{\phi}  \geq \delta \omega_X$, we thus have 
  \begin{equation}\label{eq: non vanishing}
 dd^c (\omega_u^j \wedge \omega_{\phi}^{k-j-1} \wedge \omega_X^{n-k}) \leq C_2 \sum_{l=0}^{k}  \omega_u^l \wedge \omega_{\phi}^{k-l} \wedge \omega_X^{n-k}.  	
  \end{equation}
 Since $u\leq \phi$, Lemma \ref{lem:Demkey} ensures that
	\[
	{\bf 1}_{\{u=\phi\}} \omega_{u}^{j}\wedge \omega_{\phi}^{k-j} \wedge \omega_X^{n-k} \geq {\bf 1}_{\{u=\phi\}} \omega_u^{k} \wedge \omega_{\phi}^{n-k}. 
	\]
	Noting that $X\setminus U = \{u=\phi\}$ we can write 
	\begin{flalign*}
		 \int_U (\omega_{u}^{j}\wedge\omega_{\phi}^{k-j}\wedge\omega_X^{n-k}  - \omega_u^k\wedge \omega_{X}^{n-k})
		&  \leq \int_X (\omega_{u}^{j}\wedge\omega_{\phi}^{k-j}\wedge\omega_X^{n-k}  - \omega_u^k\wedge \omega_{X}^{n-k})\\
		&= \int_X dd^c (\phi-u) \wedge \sum_{l=0}^{k-j}\omega_{u}^{l} \wedge \omega_{\phi}^{k-j-l}\wedge \omega_X^{n-k} \\
		&= \int_X (\phi-u) dd^c \left[\sum_{l=0}^{k-j}\omega_{u}^{l} \wedge \omega_{\phi}^{k-j-l}\wedge \omega_X^{n-k}\right]. 
	\end{flalign*}
	In the last line above, we have used Stokes theorem.  
	Since $0\leq \phi-u\leq s$, using \eqref{eq: non vanishing} we  can continue the above inequalities as follows: 
	\begin{flalign*}
		\int_U (\omega_{u}^{j}\wedge\omega_{\phi}^{k-j}\wedge\omega_X^{n-k}  - \omega_u^k\wedge \omega_{X}^{n-k})& \leq C_3s  \sum_{l=0}^k\int_U \omega_{u}^{l} \wedge \omega_{\phi}^{k-l}\wedge \omega_X^{n-k}. 
	\end{flalign*}
	Summing up the above inequalities for $j=0,...,k$, one gets 
	\[
	(1-C_4s)  \sum_{l=0}^k\int_U \omega_{u}^{l} \wedge \omega_{\phi}^{k-l}\wedge \omega_X^{n-k} \leq (k+1)\int_U \omega_u^k \wedge \omega_X^{n-k}. 
	\]
	If $\omega_{u}^k \wedge \omega_X^{n-k}=0$ then, for $s>0$ small enough,  the above inequality yields $\int_U \omega_{\phi}^k \wedge \omega_X^{n-k}=0$, hence $U$ has Lebesgue measure $0$, which is impossible. 
 \end{proof}

 
 
  \begin{theorem} \label{thm: domination principle}
 	Assume $\omega$ is hermitian, $u,v$ are continuous $(\omega,\omega_X,k)$-subharmonic functions
 	 such that $\omega_u^k\wedge \omega_X^{n-k} =0$ on $\{u<v\}$. Then $u\geq v$.
 \end{theorem}
 
 \begin{proof}
 	By replacing $v$ with $\max(u,v)$ and noting that $\{u<v\} = \{u<\max(u,v)\}$, we can assume that $u\leq v$. Our plan is thus to show that $u=v$.
 	
 	For $b\geq 1$, we set $\phi_b : = P_{\omega,k}(bu-(b-1)v)$. It is a negative $(\omega,k)$-sh function. Since $bu-(b-1)v$ is continuous on $X$, so is $\phi_b$. Observe also that 
 	$$
 	b^{-1}\phi_b+ (1-b^{-1})v\leq u
 	$$
 	 with equality on the contact set $\mathcal C:= \{\phi_b=bu-(b-1)v\}$.
 	 Now by Lemma \ref{lem:Demkey}
 	\[
 	{\bf 1}_{\mathcal C} \left(b^{-k} \omega_{\phi_b}^k \wedge \omega_X^{n-k} + (1-b^{-1})^k \omega_v^k\wedge \omega_X^{n-k} \right) \leq {\bf 1}_{\mathcal C} \omega_u^k\wedge \omega_X^{n-k},
 	\]
while $\int_{\mathcal C} H_k(\phi_b)=\int_X H_k(\phi_b)>0$ by Theorem \ref{thm: non vanishing}.
Since $H_k(u)=0$ on $\{u<v\}$, it follows that  $\mathcal C \cap \{u=v\}$ contains at least one point, say $x$. We then have 
$$
\phi_b(x) = bu(x) -(b-1)v(x) = v(x)\geq \inf_X v;
$$
 in particular $\sup_X \phi_b \geq \inf_X v$. 
 
 Note  that $\phi_b$ is decreasing in $b$ since $u\leq v$. The limit $\phi:=\lim_{b\to +\infty} \phi_b$ is a $(\omega,k)$-sh function which is not identically $-\infty$. For any $a>0$, 
 $$
 \phi_b \leq bu-(b-1)v \leq v -ba
 $$
  on the open set $\{u<v-a\}$, thus $\phi=-\infty$ on this set. Since $\phi$ is integrable on $X$, the set $\{u<v-a\}$ must be empty, hence letting $a\to 0^+$ gives $u\geq v$.  
 \end{proof}

  \begin{corollary}\label{cor: min principle}
  	Assume $\omega$ is hermitian, $u$ and $v$ are continuous $(\omega,\omega_X,k)$-subharmonic functions such that $\omega_u^k\wedge \omega^{n-k} \leq c \omega_v^k \wedge \omega^{n-k}$ on $\{u<v\}$ for some  $c\in [0,1)$. Then $u\geq v$.
  \end{corollary}

 \begin{proof}
Take $b$ so large that $(1-b^{-1})^k \geq c$, and consider $\phi:= P_{\omega,k}(bu-(b-1)v)$. Then $b^{-1}\phi+ (1-b^{-1})v\leq u$ with equality on the contact set 
\[\mathcal C:= \{\phi=bu-(b-1)v\}.
\] 
By Lemma \ref{lem:Demkey} we have 
 	\[
 	{\bf 1}_{\mathcal C} \left(b^{-k} \omega_{\phi}^k \wedge \omega_X^{n-k} + (1-b^{-1})^k \omega_v^k\wedge \omega_X^{n-k} \right) \leq {\bf 1}_{\mathcal C} \omega_u^k\wedge \omega_X^{n-k}. 
 	\]
 	Multiplying with ${\bf 1}_{\{u<v\}}$ and using the assumption $\omega_u^k\wedge \omega^{n-k} \leq c \omega_v^k \wedge \omega^{n-k}$ on $\{u<v\}$ and $(1-b^{-1})^k\geq c$,  we then see that 
 	\[
 	{\bf 1}_{\mathcal C \cap \{u<v\}}  \omega_{\phi}^k \wedge \omega^{n-k} =0. 
 	\]
 	Since $\omega_{\phi}^k \wedge \omega^{n-k}$ is supported on $\mathcal C$, we thus have 
 	\[
 	\int_{\{\phi < v\}} \omega_{\phi}^k \wedge \omega^{n-k} = \int_{\{bu-(b-1)v < v\}} \omega_{\phi}^k \wedge \omega^{n-k} =\int_{ \{u<v\}}  \omega_{\phi}^k \wedge \omega^{n-k}=0. 
 	\]
 	Invoking the domination principle (Theorem \ref{thm: domination principle}) we then get $\phi\geq v$. It thus follows that $bu-(b-1)v\geq \phi\geq v$, hence $u\geq v$.
 \end{proof}

 \begin{coro} \label{cor: max principle exp}
	Assume $\omega$ is hermitian and $u$ and $v$ are continuous 
	$(\omega,\omega_X,k)$-subharmonic functions.
	Then
$$
 e^{-v}(\omega+dd^c v)^k \wedge \omega_X^{n-k} \geq e^{-u} (\omega+dd^c u)^k \wedge \omega_X^{n-k}
 \Longrightarrow u\geq v.
$$
 \end{coro}
 
 \begin{proof}
 Fix $a>0$. On the set $\{u<v-a\}$ we have $H_k(u)\leq e^{-a} H_k(v)$. By Corollary \ref{cor: min principle}, $u\geq v-a$. Since it holds for all $a>0$, we obtain $u\geq v$. 
 \end{proof}

  \section{Proof of the main theorems}

    \subsection{A priori estimates}

In this section we prove Theorem B,
 which is an extension of the main result  of Ko{\l}odziej-Nguyen \cite{KN16}:

\begin{theorem} \label{thm:uniform}
Let $\omega_0,\omega$ be  semi-positive and big forms  such that $\omega_0 \leq \omega$.
Let $\f$ be a smooth $(\omega,\omega_X,k)$-subharmonic function satisfying
 $$
 (\omega+dd^c \f)^k \wedge \omega_X^{n-k} \leq f dV_X,
 $$
where  $f \in L^p(dV_X)$, $p>n/k$.
 Then
 $$
 {\rm Osc}_X(\f) \leq C,
 $$
 where $C>0$ depends on $\omega_0,n,k,p$ and an upper bound for
 $||f||_{L^p}$.
\end{theorem}

\begin{proof}
We can assume that 
 $\sup_X \varphi=0$. 
Let $u:= P_{\omega}(\varphi)$ be the largest $\omega$-psh function lying below $\varphi$. 
Let $\mathcal{C}$ denote the contact set $\mathcal{C}:= \{u=\varphi\}$.
 Then  $u$ is $(\omega,\omega_X,k)$-sh with $u \leq \f$
 hence Lemma \ref{lem:Demkey} yields
\[
{\bf 1}_{\mathcal{C}}(\omega+dd^c u)^k \wedge \omega_X^{n-k} \leq {\bf 1}_{\mathcal{C}} (\omega+dd^c\varphi)^k \wedge \omega_X^{n-k} 
\leq  f dV_X.
\]

Observe that the smooth function $h:=\omega_X^n/dV_X \geq \delta_0$ is positive on $X$.
Recall that the   Monge-Amp\`ere measure $(\omega+dd^c u)^n$ of $u$ is concentrated on the contact set 
${\mathcal C}$. We set $g:=(\omega+dd^c u)^n/dV_X$.
It follows from G{\aa}rding's inequalities \eqref{eq:mixed} that
$$
\delta_0^{1-k/n} g^{k/n} \leq h^{1-k/n} g^{k/n} \leq f,
$$
hence $g \leq C_0 f^{n/k}$. Since $f \in L^p$  with $p>n/k$,
we obtain $g \in L^{1+\delta}$ with $\delta>0$.
Theorem \ref{thm:uniformHessian}
therefore provides  a uniform bound  
$$
{\rm Osc}_X(u) \leq M.
$$
This provides a uniform bound on $u \leq 0$ if we can bound $\sup_X u$ from below.

Let $q'$ be the conjugate exponent to $p'=pk/n>1$.
We obtain a uniform lower bound on $\sup_X u$
by using \cite[Proposition 3.4]{GL22} and
H\"older inequality, via
\begin{eqnarray*}
v_M^-(\omega) (-\sup_X u)^{1/q'} & \leq &  \int_X |u|^{1/q'} (\omega+dd^c u)^n \\
&\leq &\left(\int_X |\f| \omega^n \right)^{1/q'}  \left(\int_X f^{p'} \omega^n \right)^{1/p'} \\
 &\leq & C,
\end{eqnarray*}
as    $\int_X |\f| \omega^n $ is bounded from above independently of $\f$ by Proposition \ref{pro:integrability}.
Together with the uniform bound for ${\rm Osc}_X(u)$, this yields a uniform bound
 for $|u|$. This implies a uniform bound for $\varphi$ since $u\leq \varphi$ and $\sup_X \varphi=0$. 
\end{proof}

 \begin{remark} \label{rem:optimalbound}
 We also have a uniform a priori bound
 when $f$ satisfies the weaker integrability condition
 $$
 \int_X f^{n/k} |\log f|^n \left| h \circ \log \circ \log f\right|^n dt <+\infty
 $$
 where $h$ is an increasing continuous function such that $\int^{+\infty} dt/h(t)<+\infty$.
It follows indeed that $g$ satisfies then Ko{\l}odziej's optimal integrability condition
$$
\int_X g |\log g|^n \left| h \circ \log \circ \log g\right|^n dt<+\infty,
$$
which  ensures a uniform bound on ${\rm Osc}_X(u)$ by \cite[Theorem 2.5.2]{Kol98}.
We actually need here the extension of Ko{\l}odziej's result to
the case of semi-positive and big forms \cite{EGZ09}, as well as to the 
hermitian setting \cite{DK12,GL21}.

The same proof thus applies with a minor modification: to obtain a lower bound for $\sup_X u$,
 we replace H\"older inequality by the additive  H\"older-Young inequality.
If $w$ denotes the convex weight $w(t) \sim t (\log t)^n (h \circ \log \circ \log t)^n$,
$w^*$ is its conjugate convex weight, and 
$(w^*)^{-1}$ denotes the inverse of $w^*$,
we obtain this way
 a uniform upper bound
\begin{eqnarray*}
 v_M^-(\omega) (w^*)^{-1}(-\sup_X u) 
 &\leq & \int_X (w^*)^{-1}(u) (\omega+dd^c u)^n \\
 &\leq & \int_X w \circ f \, dV_X+\int_X (-u) dV_X
\leq C.
\end{eqnarray*}
 \end{remark}

   \subsection{Stability estimates}

We strengthen Theorem B by establishing  the following stability estimate:
 
 \begin{theorem} \label{thm:stability}
 Assume $\omega$ is hermitian and $\omega\geq \omega_0$ for some semi-positive and big form $\omega_0$. 
  Assume   $0 \leq f, g \in L^p(dV_X)$ with $p>n/k$, $\varphi,\psi$ 
  are continuous $(\omega,\omega_X,k)$-subharmonic functions such that 
 \[
 (\omega+dd^c \varphi)^k \wedge \omega_X^{n-k} \geq e^{\varphi} f \omega_X^n, 
\; \;  \text{ and } \; \; 
 (\omega+dd^c \psi)^k \wedge \omega_X^{n-k} \leq e^{\psi} g \omega_X^n. 
 \]
 
 Then 
 \begin{equation}
 	\label{eq: stability}
 	 \varphi \leq \psi + C \|f-g\|_{p}^{1/k}
 \end{equation}
 for some  $C>0$ depending on $\omega_0, n,k$ and upper bounds 
 for $||f||_p, \|\varphi\|_{\infty}$, $|\inf_X \psi|$.
\end{theorem}
  
  In particular if $\f,\p$ are solutions of the corresponding equations, then
  $$
  ||\f-\p||_{\infty} \leq C \|f-g\|_{p}^{1/k},
  $$
so that $\f-\p$ is uniformly small if $f$ is close to $g$ in $L^p$.
 
 \begin{proof}
 Our proof uses a perturbation argument from \cite{GLZ18,LPT20}, which goes back to \cite{Kol96}. 
We solve the complex Monge-Amp\`ere equation 
\[
(\omega+dd^c u)^n = b h \omega_X^n, 
\; \;  \text{ where } \;  \; 
h=\frac{|f-g|^{n/k}}{\|f-g\|_p^{n/k}} +1,
\]
$\sup_X u=0$ and $b>0$ is a constant. 
The existence of a continuous solution $u \in PSH(X,\omega)$  follows from \cite{KN15, KN19}. 

We can assume without loss of generality that $\omega_X$ is a Gauduchon metric
(i.e. $dd^c (\omega_X^{n-1})=0$) which dominates $\omega$. 
Thus $u $ is $\omega_X$-psh with 
$$
(\omega_X+dd^c u)^n \geq (\omega+dd^c u)^n \geq b \omega_X^n
$$
 since $h \geq 1$. 
The mixed Monge-Amp\`ere inequalities (see \cite{Ng16}) yield
$$
(\omega_X+dd^c u) \wedge \omega_X^{n-1} \geq b^{1/n} \omega_X^n,
$$
hence $b \leq 1$ since $\int_X (\omega_X+dd^c u) \wedge \omega_X^{n-1}=\int_X \omega_X^n$
by the Gauduchon condition.
 
Since $h \in L^q$ with $q= pk/n>1$ and $b \leq 1$,
 Theorem \ref{thm:uniform} ensures that $\|u\|_{\infty}$ is uniformly bounded.  
 It thus follows from \cite[Proposition 3.4]{GL22} that $b \geq \delta_0>0$
 is uniformly bounded from below.

If $\|f-g\|_{p}^{1/k}$ is not  small, the stability inequality \eqref{eq: stability} is trivially
obtained by adjusting the value of $C$. In the remainder of the proof, we can thus assume without
loss of generality that 
\begin{eqnarray} \label{eq:stab2}
\e:=b^{-1/n} \|f-g\|_p^{1/k}e^{\sup_X \varphi/k}  \leq 1/2.
\end{eqnarray}

By the mixed Monge-Amp\`ere inequalities \eqref{eq:mixed} we have 
\[
(\omega+dd^c u)^k \wedge \omega_X^{n-k}\geq b^{k/n} \frac{|f-g|}{\|f-g\|_p}\omega_X^n. 
\]

We set $C_1= 2k - \inf_X \varphi$ and
$$
\psi_{\varepsilon} := (1-\varepsilon) \varphi + \varepsilon u  - C_1 \varepsilon,
$$
The function $\p_{\e}$ belongs to $SH(X,\omega,k) \cap {\mathcal C}^0(X)$;
 its Hessian measure satisfies
\begin{eqnarray*}
(\omega+dd^c \psi_{\varepsilon})^k \wedge \omega_X^{n-k} 
&\geq & (1-\varepsilon)^k e^{\varphi} f \omega_X^n + \varepsilon^k b^{k/n}  \frac{|f-g|}{\|f-g\|_p} \omega_X^n \\
& = & \left[ (1-\varepsilon)^k e^{\varphi} f +e^{\sup_X \f} |f-g| \right] \omega_X^n \\
& \geq & e^{\f-2k\e} g \omega_X^n,
\end{eqnarray*}
using \eqref{eq:stab2}, $|f-g|+f \geq g$ and 
the inequality $\log(1-\e) \geq -2\e$ valid on $[0,1/2]$.

Observe now that $\p_{\e} \leq \f -\e \f-C_1 \e \leq \f-2k\e$
by the definition of $C_1$ and $u \leq 0$. We infer
\[
(\omega+dd^c \psi_{\varepsilon})^k \wedge \omega_X^{n-k} \geq e^{\psi_{\varepsilon}} g \omega_X^n,
\]
hence  $\psi_{\varepsilon}\leq \psi$ by Corollary \ref{cor: max principle exp}. 
This yields 
$$
\f \leq \p+\e(\f-u)+C_1 \e \leq \p+C_2 ||f-g||_p^{1/k},
$$
for some uniform constant $C_2$, as desired. 
 \end{proof}

    \subsection{Proof of Theorem A}

We approximate $f$ in $L^p$ by  smooth and positive densities $f_{j}$.
For $j\geq 1$ we  consider the hermitian forms $\omega_j=\omega+2^{-j} \omega_X$.
It follows from \cite{Zh17,Szek18} that there exist 
 unique constants $c_j>0$ and 
smooth $(\omega_j,k)$-sh functions $\f_{j}$
 such that
\[
(\omega_j+dd^c \f_j)^k \wedge \omega_X^{n-k}=c_j f_{j} \omega_X^n,
\]
and $\sup_X \f_j=0$. 

We claim that the constants $c_{j}$ stay uniformly bounded,
as $j$ increases to $+\infty$. To prove the claim we use the solutions to the Monge-Amp\`ere equations by Tosatti-Weinkove \cite{TW10}: 
\[
(\omega_j+dd^c u_j)^n = b_j f_j^{n/k} \omega_X^n. 
\]

Set $p' = pk/n>1$. Since the $L^{p'}$ norm of $f_j^{n/k}$ is uniformly bounded from above, by \cite[Lemma 3.3]{GL21}, there exists a constant $c>0$ independent of $j$ and a bounded $\omega$-psh function $u$ such that $\omega_u^n\geq c f_j^{n/k}$. By \cite[Corollary 1.13]{GL21}, $b_j\geq c$. It then follows from G{\aa}rding's inequality that  
\[
(\omega_j+dd^c u_j)^k \wedge \omega_X^{n-k}\geq c^{k/n} f_j \omega_X^n. 
\]
Corollary \ref{cor: min principle} thus ensures that $c_j \geq c^{k/n}$ is uniformly bounded from below by a positive constant.  

We next bound $c_j$ from above. By G{\aa}rding's inequality we have
\[
(\omega_j +dd^c \varphi_j)\wedge \omega_X^{n-1}\geq c_j^{1/k} f_j^{1/k} \omega_X^n.
\]
The functions $\f_{j}$ belong to $SH(X,\omega_1)$ and are relatively compact in $L^1$, hence integrating over $X$ and using that $\int_X f_j^{1/k}\omega_X^n \to \int_X f^{1/k}\omega_X^n>0$, we obtain a uniform upper bound for $c_j$. 

\smallskip

Since $\omega \leq \omega_{j}$ and $||c_{j}f_{j}||_{L^p}$ is uniformly bounded from above,
 it follows  from Theorem \ref{thm:uniform} that the $\f_j$'s are uniformly bounded: 
 \[
 -C_1 \leq \varphi_j \leq 0,
 \]
 for a uniform constant $C_1$. 
 Extracting and using compactness in $L^1(X)$ of  normalized $(\omega,\omega_X,k)$-sh functions,
we can thus assume that $c_j\rightarrow c > 0$ and $\f_j$ converges in $L^1(X)$ and almost everywhere 
to a bounded function $\f$ which is $(\omega,\omega_X,k)$-sh, as $j\to +\infty$.

 We next prove that the convergence $\varphi_j\to \varphi$ is locally uniform in $\Omega$. 
We set $g_j := c_j f_j e^{-\varphi_j}$ and $g=c f e^{-\varphi}$. 
By H\"older inequality, $\|g_j-g\|_q\to 0$ for some $q>n/k$.
 Fix  $j<l$ large enough. Since $\omega_j\geq \omega_l$, we have 
\[
(\omega_j +dd^c \varphi_l)^k \wedge \omega_X^{n-k} \geq e^{\varphi_l} g_l \omega_X^n.
\] 
By Theorem \ref{thm:stability}, we infer 
\begin{equation}\label{eq: solution}
	\varphi_l \leq \varphi_j + C_2 \|g_j-g_l\|_q^{1/k}, 
\end{equation}
for a uniform constant $C_2$. We next consider 
\[
\psi_{j} : = (1-\varepsilon)\varphi_j  +\varepsilon \rho -(2k+C_1) \varepsilon, 
\]
where $\varepsilon =\delta^{-1}2^{-j}<1/2$ when $j$ is large enough. Since $\omega+dd^c \rho \geq \delta \omega_X$, we have 
\begin{flalign*}
	\omega_l + dd^c \psi_j &  \geq  \omega + \varepsilon dd^c \rho + (1-\varepsilon) dd^c \varphi_j \\
	& \geq \varepsilon \delta \omega_X+ (1-\varepsilon) (\omega_j + dd^c \varphi_j) - 2^{-j} \omega_X\\
	& = (1-\varepsilon) (\omega_j+dd^c \varphi_j). 
\end{flalign*}
It follows that $\psi_j \in SH(X,\omega_l,k)$ and 
\[
(\omega_l +dd^c \psi_j)^k \wedge \omega_X^{n-k} \geq e^{\varphi_j + k\log(1-\varepsilon)} g_j \omega_X^n\geq e^{\varphi_j-2k \varepsilon} g_j \omega_X^n \geq e^{\psi_j} g_j \omega_X^n
\]
holds in $\Omega$. The function $\psi_{j,l}:= \max(\psi_j, \varphi_l)$ is $(\omega_l,k)$-subharmonic and continuous on $X$. Lemma \ref{lem:Demkey} yields
\[
(\omega_l +dd^c \psi_{j,l})^k \wedge \omega_X^{n-k} \geq e^{\psi_{j,l}} \min(g_j,g_l) \omega_X^n. 
\]
Applying Theorem \ref{thm:stability} we then get 
\[
\psi_{j,l} \leq \varphi_l + C_3 \|\min(g_j,g_l)-g_l\|_q^{1/k} \leq \varphi_l + C_3\|g_j-g_l\|_q^{1/k},
\]
for a uniform constant $C_3$. Together with \eqref{eq: solution}, this implies 
\begin{eqnarray*}
\varphi_j + \delta^{-1}2^{-j}\rho -(2k+C_1) \delta^{-1} 2^{-j} 
&\leq & \varphi_l + C_3 \|g_j-g_l\|_q^{1/k} \\
&\leq & \varphi_j+C_1\|g_j-g_l\|_q^{1/k}.  
\end{eqnarray*}
Therefore, $\varphi_j$ converges locally uniformly in $\Omega$ to $\varphi$. 
Thus $\varphi$ is continuous in $\Omega$ where  it solves the complex Hessian equation 
\[
(\omega+dd^c \varphi)^k \wedge \omega_X^{n-k}= cf\omega_X^n.
\]

\begin{remark}
When the form $\omega$ is hermitian, 
the solution $\f$ that we construct is continuous on the whole of $X$. 
Continuity of $\f$ at points where $\omega$ vanishes
is a delicate issue, even in the case of  complex Monge-Amp\`ere equations.
We refer the interested reader to \cite{GGZ20} for a recent account.

The uniqueness of $\f$ is also a subtle problem in the hermitian setting, 
see \cite{KN19} for a partial result in the case of the complex Monge-Amp\`ere equation.
\end{remark}


\begin{thebibliography}{99}

  \bibitem[BT76]{BT76} E.~Bedford, B.A.~Taylor,  \emph{The Dirichlet problem for the complex Monge-Amp\`ere
operator}.  Invent. Math., 37 (1976), 1--44.

  
  \bibitem[BT82]{BT82} E.~Bedford, B.A.~Taylor,  \emph{A new capacity for plurisubharmonic  functions}. 
Acta Math. \textbf{149} (1982), no.~1-2, 1--40.  
  
   \bibitem[BZ20]{BZ20} A.Benali, A.Zeriahi,
 \emph{ The H\"older continuous subsolution theorem for complex Hessian equations}. 
 J. \'Ec. polytech. Math. 7 (2020), 981--1007. 
  
 
 
  
 \bibitem[Ber19]{Ber19} R.~J. Berman, 
 \emph{From {M}onge-{A}mp\`ere equations to envelopes and  geodesic rays in the zero temperature limit}.  
 Math. Z. \textbf{291} (2019),  no.~1-2, 365--394.  

  
  
 \bibitem[Blo05]{Blo05} Z. Blocki, 
\emph{Weak solutions to the complex Hessian equation}.
Ann. Inst. Fourier (Grenoble) 55 (2005), no. 5, 1735--1756. 

 
\bibitem[BEGZ10]{BEGZ10}
S.~Boucksom, P.~Eyssidieux, V.~Guedj,  A.~Zeriahi, 
\emph{Monge-{A}mp\`ere  equations in big cohomology classes}. 
Acta Math. \textbf{205} (2010), no.~2, 199--262. 
 
  

 \bibitem[Char16]{Char16} M.Charabati,
\emph{ Modulus of continuity of solutions to complex Hessian equations}.
Internat. J. Math. 27 (2016), no. 1, 1650003, 24 pp. 


 \bibitem[CZ19]{CZ19} J.Chu, B.Zhou,
\emph{Optimal regularity of plurisubharmonic envelopes on compact Hermitian manifolds}.
Sci. China Math. 62 (2019), no. 2, 371--380.

 \bibitem[CM21]{CM21} J.Chu, N.McCleerey
\emph{Fully non-linear degenerate elliptic equations in complex geometry}.
  J. Funct. Anal. 281 (2021), no. 9, Paper No. 109176


 \bibitem[Dem94]{Dem94} J.P. Demailly, 
 {\it Regularization of closed positive currents of type {$(1,1)$} by the flow of a {C}hern connection}. 
  Contributions to complex analysis and analytic geometry,  
  {105--126}, Friedr. Vieweg, Braunschweig (1994).
  
   
     \bibitem[DDT19]{DDT19}S.~Dinew, H.S.Do, T.D.T\^o,
   \emph{  A viscosity approach to the Dirichlet problem for degenerate complex Hessian-type equations}. 
Anal. PDE 12 (2019), no. 2, 505--535. 

 \bibitem[DK12]{DK12} S.~Dinew, S.Ko{\l}odziej,
   \emph{ Pluripotential estimates on compact Hermitian manifolds}. 
   Advances in geometric analysis, 69-86, Adv. Lect. Math. (ALM), 21, Int. Press, 2012. 
  
  \bibitem[DK14]{DK14}S.~Dinew, S.Ko{\l}odziej,
  \emph{A priori estimates for complex Hessian equations}. 
Anal. PDE 7 (2014), no. 1, 227--244. 

  \bibitem[DK17]{DK17} S.~Dinew, S.Ko{\l}odziej,
 \emph{Liouville and Calabi-Yau type theorems for complex Hessian equations}. 
 Amer. J. Math. 139 (2017), no. 2, 403--415. 
 
   \bibitem[DL15]{DL15} S.~Dinew, C.H~Lu,  \emph{ Mixed Hessian inequalities and uniqueness in the class $\mathcal{E}(X,\omega,m)$}.   Math. Z. 279 (2015), no. 3-4, 753--766
 
  

\bibitem[EGZ09]{EGZ09} P.~Eyssidieux, V.~Guedj, A.~Zeriahi, 
\emph{Singular {K}\"{a}hler-{E}instein  metrics}. 
J. Amer. Math. Soc. \textbf{22} (2009), no.~3, 607--639.

\bibitem[Gar59]{Gar59} L. G{\aa}rding, \emph{An inequality for hyperbolic polynomials}. J. Math. Mech. 8 (1959) 957--965.  



\bibitem[GGZ23]{GGZ20} V.~Guedj, H.~Guenancia, A.~Zeriahi, 
\emph{Continuity of singular K\"ahler-Einstein potentials}. 
  Int. Math. Res. Not. (2023) no. 2, 1355--1377.


\bibitem[GL22]{GL22} V.~Guedj, C.~H. Lu, 
\emph{Quasi-plurisubharmonic envelopes 2: bounds on Monge-Amp\`ere volumes}. 
  Algebr. Geom. 9 (2022), no. 6, 688--713.

\bibitem[GL21]{GL21} V.~Guedj, C.~H. Lu, 
\emph{Quasi-plurisubharmonic envelopes 3: Solving Monge-Ampère equations on hermitian manifolds}. 
Preprint (2021).
 
 \bibitem[GLZ18]{GLZ18} V.~Guedj, C.~H. Lu, A.~Zeriahi, \emph{Stability of solutions to complex Monge-Amp\`ere flows}.  Ann. Inst. Fourier (Grenoble) 68 (2018), no. 7, p. 2819--2836.
 
\bibitem[GLZ19]{GLZ19} V.~Guedj, C.~H. Lu, A.~Zeriahi, 
\emph{Plurisubharmonic envelopes and supersolutions}. 
J. Differential Geom. \textbf{113} (2019), no.~2, 273--313.
   
 
 
 
\bibitem[GZ17]{GZbook} V.~Guedj, A.~Zeriahi, 
\emph{Degenerate complex {M}onge-{A}mp\`ere  equations}. 
EMS Tracts in Mathematics, vol.~26, European Mathematical Society  (EMS), Z\"{u}rich, 2017. 
  
  \bibitem[GN18]{GN18}   D. Gu, N.C. Nguyen,
\emph{ The Dirichlet problem for a complex Hessian equation on compact Hermitian manifolds with boundary}. 
Ann. Sc. Norm. Sup. Pisa  (5) 18 (2018),  1189--1248. 
  
  
   \bibitem[GP22]{GP22} B. Guo, D.H. Phong,  
  \emph{On $L^{\infty}$ estimates for fully nonlinear partial differential equations on Hermitian manifolds}.  Preprint arXiv (2022).
  
   \bibitem[GPT21]{GPT21} B. Guo, D.H. Phong, F. Tong,
  \emph{On $L^{\infty}$ estimates for complex Monge-Amp\`ere equations}.
  Preprint arXiv (2021).
  
   \bibitem[GPTW21]{GPTW21} B.Guo, D.H. Phong, F. Tong, C. Wang,
  \emph{On $L^{\infty}$ estimates for complex Monge-Amp\`ere equations and Hessian equations on nef classes}.
  Preprint arXiv (2021).
  
  
  \bibitem[HH72]{HH72} M.~Herv\'e, R.-M.~Herv\'e, \emph{Les fonctions surharmoniques dans l'axiomatique de M. Brelot associ\'ees \`a un op\'erateur elliptique d\'eg\'en\'er\'e}. 
Ann. Inst. Fourier, tome 22 (1972),131--145.

\bibitem[HL13]{HL13} R.~Harvey, B.~Lawson, \emph{
The equivalence of viscosity and distributional subsolutions for convex subequations, a strong Bellman principle}. 
Bull. Braz. Math. Soc. 44 (2013),  621--652. 
 
 \bibitem[Kol96]{Kol96} S.~Ko{\l}odziej, \emph{Some sufficient conditions for solvability of the Dirichlet problem for the complex Monge-Amp\`ere operator}.  Ann. Polon. Math. 65 (1996), no. 1,
p. 11--21.
  
  \bibitem[Kol98]{Kol98} S.~Ko{\l}odziej, 
  \emph{The complex {M}onge-{A}mp\`ere equation}. 
  Acta Math.  \textbf{180} (1998), 69--117.
  
  \bibitem[KN15]{KN15} S.~Ko{\l}odziej, N.C.Nguyen,
\emph{ Weak solutions to the complex Monge-Amp\`ere equation on compact Hermitian manifolds}. 
Contemp. Math. 644 (2015), 141--158
  
    \bibitem[KN16]{KN16} S.~Ko{\l}odziej, N.C.Nguyen,
  \emph{ Weak solutions of complex Hessian equations on compact Hermitian manifolds}. 
Compos. Math. 152 (2016), no. 11, 2221--2248. 


   \bibitem[KN19]{KN19} S.~Ko{\l}odziej, N.C.Nguyen,
 \emph{ Stability and regularity of solutions of the Monge-Ampère equation on Hermitian manifolds}. 
Adv. Math. 346 (2019), 264--304. 

\bibitem[Lit83]{Lit83} W.~Littman, \emph{Generalized subharmonic functions: Monotonic approximations and an improved maximum principle}. 
Ann. Scuola Norm. Sup. Pisa Cl. Sci. (3) 17 (1963), 207--222. 
 

\bibitem[Lu13]{Lu13} C.~H. Lu, 
\emph{Viscosity solutions to complex Hessian equations}. 
J. Funct. Anal. 264 (2013), no. 6, 1355--1379. 

\bibitem[Lu15]{Lu15} C.~H. Lu, 
\emph{A variational approach to complex {H}essian equations in  {$\mathbb{C}^n$}}. 
J. Math. Anal. Appl. \textbf{431} (2015), no.~1, 228--259. 
 
  
  \bibitem[LN15]{LN15} C.~H. Lu, V.D. Nguyen,
  \emph{Degenerate complex Hessian equations on compact K\"ahler manifolds}. 
 Indiana Univ. Math. J. 64 (2015), no. 6, 1721--1745. 
 
 

   \bibitem[LPT21]{LPT20} C.~H. Lu, T.T.Phung, T.D.To,
  \emph{ Stability and H\"older regularity of solutions to complex Monge-Amp\`ere equations on compact hermitian manifolds}. 
Ann. Inst. Fourier 71 (2021), no. 5, 2019--2045. 

\bibitem[Mic82]{Mic82} M.L. Michelsohn, \emph{On the existence of special metrics in complex geometry}. Acta Math. 149 (1982), 261--295.

\bibitem[Ng16]{Ng16} N.-C.~Nguyen, \emph{The complex Monge-Amp\`ere type equation on compact Hermitian manifolds and applications}.  Adv. Math. 286 (2016), p. 240--285.
 
 
  \bibitem [Szek18]{Szek18} G. Sz\'ekelyhidi,  
 \emph{ Fully non-linear elliptic equations on compact Hermitian manifolds}.  
 J. Differential Geom. 109 (2018), no. 2, 337--378.
 
  
   \bibitem [Tos18]{Tos18} V.Tosatti,  
\emph{  Regularity of envelopes in K\"ahler classes}. 
Math. Res. Lett. 25 (2018), no1,  281--289.
 
  \bibitem [TW10]{TW10} V.Tosatti, B. Weinkove,
\emph{  The complex Monge-Ampère equation on compact Hermitian manifolds}. 
 J. Amer. Math. Soc. 23 (2010), no. 4, 1187--1195.
 
 
 
 
\bibitem [Yau78]{Yau78}  S.~T.~Yau,
\emph{On the Ricci curvature of a compact K{\"a}hler manifold and the complex Monge-Amp{\`e}re equation. I}. 
 Comm. Pure Appl. Math. {\bf 31} (1978), no. 3, 339--411.  
 
 
   \bibitem [Zh17]{Zh17} D.Zhang,
\emph{ Hessian equations on closed Hermitian manifolds}.
  Pacific J. Math. 291 (2017), no. 2, 485--510.
 
  

\end{thebibliography}
\end{document}